\documentclass[reqno,12pt]{amsart}%
\usepackage[body={17cm,21cm}]{geometry}

\usepackage[T1]{fontenc}
\usepackage{graphicx}
\usepackage{mathrsfs}
\usepackage{amscd,latexsym,amsthm,amsfonts,amssymb,amsmath,amsxtra}
\usepackage[colorlinks, urlcolor=blue,  citecolor=blue]{hyperref}
\usepackage{enumerate}
\usepackage[all]{xy}%
\providecommand{\U}[1]{\protect\rule{.1in}{.1in}}
\RequirePackage{amsmath}
\RequirePackage{amssymb}

\newcommand{\Gal}{{\mathrm{Gal}}}

\newcommand{\ord}{{\mathrm{ord}}}
\newcommand{\rank}{{\mathrm{rank}}}

\newcommand{\Pic}{\mathrm{Pic}}

\renewcommand{\mod}{\ \mathrm{mod}\ }

\numberwithin{equation}{section}

\theoremstyle{remark}

\newtheorem{defi}{\rm{\textbf{Definition}}}[section]

\newtheorem{exam}[defi]{\rm{\textbf{Example}}}

\newtheorem{rem}[defi]{\rm{\textbf{Remark}}}

\theoremstyle{plain}

\newtheorem{thm}[defi]{\rm{\textbf{Theorem}}}

\newtheorem{cor}[defi]{\rm{\textbf{\textbf{Corollary}}}}

\newtheorem{lem}[defi]{\rm{\textbf{Lemma}}}

\newtheorem{prop}[defi]{\rm{\textbf{\textbf{Proposition}}}}

\begin{document}

  \title[On indefinite and potentially universal quadratic forms over number fields] {On indefinite and potentially universal quadratic forms over number fields}

 \author{Fei Xu}
 \address{School of Mathematical Science \\ Capital Normal University \\ Beijing 100048, China}
 \email{xufei@math.ac.cn}

 \author{Yang Zhang}
 \address{School of Mathematical Science \\ Capital Normal University \\ Beijing 100048, China}
 \email{2170502061@cnu.edu.cn}

 \date{}

\begin{abstract}  A number field $k$ admits a binary integral quadratic form which represents all integers locally but not globally if and only if the class number of $k$ is bigger than one. In this case, there are only finitely many classes of such binary integral quadratic forms over $k$. 

A number field $k$ admits a ternary integral quadratic form which represents all integers locally but not globally if and only if the class number of $k$ is even. In this case, there are infinitely many classes of such ternary integral quadratic forms over $k$.

An integral quadratic form over a number field $k$ with more than one variables represents all integers of $k$ over the ring of integers of a finite extension of $k$ if and only if this quadratic form represents $1$ over the ring of integers of a finite extension of $k$.
 \end{abstract}

\maketitle

\section{Introduction}

Universal integral quadratic forms over $\Bbb Z$ have been studied extensively by Dickson \cite{Di} and Ross \cite{Ro} for both positive definite and indefinite cases in the 1930s.  A positive definite integral quadratic form is called universal if it represents all positive integers. Ross pointed out that there is no ternary positive definite universal form over $\Bbb Z$ in \cite{Ro} (see a more conceptual proof in \cite[Lemma 3]{EK}).
However, this result is not true over totally real number fields. In \cite{Ma}, Maass proved that $x^2+y^2+z^2$ is a 
positive definite universal quadratic form over the ring of integers of $\Bbb Q(\sqrt{5})$. Chan, Kim and Raghavan further determined all possible quadratic real fields which admit a positive definite universal ternary quadratic form and listed all positive definite universal ternary quadratic forms in \cite{CKR}. A symbolic achievement for positive definite universal quadratic forms over $\Bbb Z$ is the 15-theorem first proved by Conway and Schneeberger and simplified by Bhargava in \cite{Bh}, which claims that a positive definite classical integral quadratic form $f$ is universal if and only if $f$ represents all positive integers from 1 to 15.
Lee in \cite{Lee} established an analogy of Conway and Schneeberger's 15-theorem over the ring of integers of $\Bbb Q(\sqrt{5})$. In fact, for any fixed totally real field, universal quadratic forms always exist by \cite{HKK}. By the reduction theory, Kim and Earnest \cite[Corollary 1]{Ea} showed that 
there are only finitely many classes of positive definite quadratic forms over a given totally real field with the minimal rank. The natural problems remain to determine the minimal rank for existing positive definite universal quadratic forms and classify all universal forms of the minimal rank over a given totally real field. Indeed, 
there are many articles concerning these subjects such as \cite{BK}, \cite{De}, \cite{EK}, \cite{KS}, \cite{Ki}, \cite{Sa}, \cite{Ya} etc. 

Indefinite quadratic forms have significantly different features from positive definite quadratic forms. An indefinite integral quadratic form is called universal if it represents all integers. Since the representability of indefinite quadratic forms with more than three variables satisfies the local-global principle by strong approximation for spin groups, the universal property of indefinite quadratic forms with more than three variables satisfies the local-global principle. By inspecting the integral spinor norms, one concludes that the class number of indefinite universal ternary quadratic forms over $\Bbb Z$ is one. Therefore the universal property of indefinite ternary quadratic forms over $\Bbb Z$  also satisfies the local-global principle. However, the local-global principle of universal property is not true any more over a number field because the class group of the ground number field will be involved. Indeed, Estes and Hsia \cite{EH} have determined all imaginary quadratic fields such that $x^2+y^2+z^2$ is locally universal but not globally. It is natural to ask what kind of number fields admit a binary or ternary integral quadratic form which is locally universal but not globally. We give a complete answer to this question (see Theorem \ref{infinite}).

\begin{thm}\label{m-1}  Let $k$ be a number field.

(1)  There is a binary integral indefinite quadratic form which is locally universal but not globally over $k$ if and only if the class number of $k$ is bigger than one. In this case, there are only finitely many classes of such binary integral quadratic forms. 

(2) There is a ternary integral indefinite quadratic form which is locally universal but not globally over $k$ if and only if the class number of $k$ is even. In this case, there are infinitely many classes of such ternary integral quadratic forms.

 \end{thm}

\begin{rem} Throughout this paper, an integral quadratic form  means a quadratic form with integral coefficients as a polynomial but not necessarily a classical integral quadratic form (with an integral matrix). By Theorem \ref{infinite} (1), there is no binary universal classical integral quadratic forms. The integral indefinite ternary quadratic forms constructed in the proof of  Theorem \ref{infinite} (2) are not classically integral. We will investigate the corresponding result for classical integral quadratic forms in future.  
\end{rem}

Dickson in \cite[Theorem 22]{Di} (see also \cite[Theorem 5]{Ro}) proved that a ternary indefinite universal quadratic form over $\Bbb Z$ has non-trivial zero over $\Bbb Q$. Ross in \cite[Theorem 6]{Ro} further proved that every ternary quadratic form with non-trivial zero over $\Bbb Q$ and square-free discriminant is universal over $\Bbb Z$. We will point out that Dickson's result still holds over number fields (see Proposition \ref{ternary-glo}) but Ross' result does not hold over general number fields any more (see Example \ref{ross-glo}).  Moreover, we also show that one can not expect the result like Conway and Schneeberger's 15-theorem for indefinite universal quadratic forms (see Proposition \ref{counter-glo}).

Since the universal property for both positive definite case and indefinite case is not stable under base change, we consider the following generalization of universal quadratic forms.

\begin{defi} Let $f$ be an integral quadratic form $f$ over the ring of integers $\frak o_k$ of a number field $k$.

i)  We say $f$ is potentially universal over $k$ if there is a finite extension $K/k$ such that every element in $\frak o_k$ can be represented by $f$ over the ring of integers of $K$. 

ii)  We say an element $\alpha$ in $\frak o_k$ is potentially represented by $f$ over $k$ if $\alpha$ can be represented by $f$ over the ring of integers of a finite extension $K/k$.
\end{defi}

One has the following theorem.

\begin{thm} \label{m-2} Let $f$ be an integral quadratic form $f$ over the ring of integers $\frak o_k$ of a number field $k$ with more than one variables. Then $f$ is potentially universal over $k$ if and only if  1 is potentially represented by $f$ over $k$. 
\end{thm}

An immediate consequence of this theorem is that the potentially universal property is stable under base change (see Corollary \ref{st}).

Terminology and notation are standard if not explained and are adopted from \cite{OM}. Let $k$ be a number field and $\frak o_k$ be the ring of integers of $k$. Let $\Omega_k$ be the set of all primes of $k$ and $\infty_k$ be the set of all archimedean primes. For any $v\in \Omega_k$, we denote $k_v$ the completion of $k$ with respect to $v$. When $v\in \Omega_k\setminus \infty_k$, the completion of $\frak o_k$ with respect to $v$ inside $k_v$ is denoted by $\frak o_{k_v}$ and the group of units of $\frak o_{k_v}$ is denoted by $\frak o_{k_v}^\times$. We fix a prime element $\pi_v$ of $\frak o_{k_v}$ for each $v\in \Omega_k\setminus \infty_k$. A non-degenerated quadratic space $(V, Q)$ is defined to be a vector space $V$ over $k$ with a map $Q: V\rightarrow k$ such that $$B(x,y)=\frac{1}{2}(Q(x+y)-Q(x)-Q(y))$$ is a non-degenerated bilinear map on $V\times V$.  Write $$ O(V)=\{ \sigma \in GL(V): \ Q(\sigma x) = Q(x),  \ \forall x\in V \}  \ \ \ \text{and} \ \ \ O^+(V)= \{ \sigma \in O(V): \ \det(\sigma)=1 \} . $$
The adelic group of $O(V)$ and $O^+(V)$ are denoted by $O_\Bbb A(V)$ and $O^+_\Bbb A(V)$ respectively.  

A finitely generated $\frak o_k$-module $L$ is called an $\frak o_k$-lattice if $(kL, Q)$ is a non-degenerated quadratic space over $k$.
We use $\frak s(L)$ and $\frak n(L)$ to denote the fractional ideals of $\frak o_k$ generated by $B(x, y)$ with $x,y\in L$ and $Q(x)$ with $x\in L$ respectively. Write 
$$Q(L) = \{ Q(x) : \ x\in L \} .$$
Similarly, for any $\frak o_k$-lattice $L$, we write $L_v$ for the completion of $L$ with respect to $v\not\in \infty_k$ and for $k_v L$ with $v\in \infty_k$. Moreover, we also have an $\frak o_{k_v}$-lattice with the same notions as above for a non-archimedean prime $v$ of  $k$. For $\alpha, \beta, \gamma  \in \frak o_{k_v}$, we denote $\gamma A(\alpha, \beta)$ a binary lattice $$ L=\frak o_{k_v} x + \frak o_{k_v} y  \ \ \ \text{with} \ \ \ Q(x)=\alpha \gamma, \ \ B(x,y)=\gamma \ \ \text{and} \ \ Q(y) = \beta \gamma. $$

The paper is organized as follows. In \S 2, we classify all local universal lattices of arbitrary rank over non-dyadic local fields and of rank less than 4 over dyadic local fields. Then we apply these result to study indefinite universal quadratic forms over number fields and give a proof of Theorem \ref{m-1} in \S 3. We prove Theorem \ref{m-2} in \S 4.

\section{Locally universal conditions}\label{local}

In this section, we study the local universal property over a local field. 

\begin{defi} An $\frak o_{k_v}$-lattice $L$ is called universal if $Q(L)=\frak o_{k_v}$ for $v\in \Omega_k\setminus \infty_k$.

\end{defi}

One has the following criterion to verify the local universal lattices.

\begin{lem} \label{cri} Let $L$ be an $\frak o_{k_v}$-lattice. Then $L$ is universal if and only if $$(\frak o_{k_v}^\times \cup \pi_v \frak o_{k_v}^\times)  \subseteq Q(L) \subseteq \frak o_{k_v} . $$
\end{lem}
\begin{proof} Since $(\frak o_{k_v}^\times \cup \pi_v \frak o_{k_v}^\times)  \subseteq Q(L)$, one gets that $\pi_v^k \frak o_{k_v}^\times   \subseteq Q(L)$ for any integer $k\geq 0$. 
Since $$\frak o_{k_v}\setminus \{0 \} =\bigcup_{k\geq 0}  \pi_v^k \frak o_{k_v}^\times , $$  one concludes $Q(L)= \frak o_{k_v}$ as desired. 
\end{proof}

\subsection{Non-dyadic cases}

By using Lemma \ref{cri}, we can determine the universal lattices in term of Jordan splittings over non-dyadic local fields. 

\begin{prop}\label{nondyadic} Let $L$ be an $\frak o_{k_v}$-lattice $L$ with the Jordan splitting $$L= L_1\perp L_2\perp \cdots \perp L_t $$ over a non-dyadic local field $k_v$. 
Then $L$ is universal if and only if $L_1$ is unimodular such that one of the following conditions holds

(1)  $\rank(L_1)\geq 3$ or $L_1=A(0,0)$.

(2) $\rank (L_1)=2$ with $L_1\ncong A(0,0)$ and $L_2$ is $(\pi_v)$-modular with $rank (L_2)\geq 2$.

\end{prop}

\begin{proof} Sufficiency. For (1), it is clear that $A(0,0)$ is universal. If $\rank(L_1)\geq 3$, then $Q(L_1)=\frak o_{k_v}$ by \cite[92:1b]{OM}. For (2), one has $$Q(L_1)\supseteq \frak o_{k_v}^\times \ \ \ \text{and} \ \ \ Q(L_2)\supseteq \pi_v \frak o_{k_v}^\times $$ by \cite[92:1b]{OM}. The result follows from Lemma \ref{cri}.

Necessity. Since $L$ is universal, one has $\frak n(L)=\frak s(L)= \frak o_{k_v}$. This implies that $L_1$ is unimodular.  If $\rank (L_1)=1$, then $\frak o_{k_v}^\times \nsubseteq Q(L)$. If $L_2$ is not $(\pi_v)$-modular, then $\pi_v\not\in Q(L)$. If $\rank (L_1)=2$ with $L_1\ncong A(0,0)$ and $\rank (L_2)=1$, then $\pi_v\frak o_{k_v}^\times \nsubseteq Q(L)$. 
\end{proof}

\begin{rem} Since Condition (2) in Proposition \ref{nondyadic} is not stable under base change, the universal property for both definite positive and indefinite cases is not stable under base change. 
\end{rem}

\subsection{Dyadic cases of general rank}

It is more complicated to determine the universal lattices over dyadic local fields. One can compare this with computing integral spinor norms over dyadic local fields. Based on various efforts of \cite{Hs}, \cite{Xu1}, \cite{Xu2} and \cite{Xu3},  Beli completed the computation of integral spinor norms over dyadic local fields in \cite{Be} in terms of BONGs. The complete determination of integral spinor norms over dyadic local fields in terms of Jordan splittings was given in \cite{LX}. Because of this complexity, we will only provide the complete solution of universal ternary quadratic forms which is sufficient for our global application. 

Let's recall a minimal norm Jordan splitting of lattices over a dyadic local field. 

\begin{defi} Let $L$ be a lattice over a dyadic local field $k_v$. A Jordan splitting 
\begin{equation} \label{jordan}  L= L_1\perp L_2\perp \cdots \perp L_t  \end{equation} is called a minimal norm Jordan splitting if for any Jordan splitting 
$$ L=L_1'\perp L_2'\perp \cdots \perp L_t' , $$ one has 
$\frak n(L_i)\subseteq \frak n(L_i')$ and $\perp_j \pi_v^{s_i}A(0,0)$ splits $L_i$ as long as $\perp_j \pi_v^{s_i}A(0,0)$ splits $L_i'$ for $1\leq i\leq t$.
\end{defi} 

By \cite[Theorem 1.1]{X-min}, such a minimal norm Jordan splitting always exists.  For a minimal norm Jordan splitting (\ref{jordan}), one can write
\begin{equation} \label{index}  i(L) = \max \{1\leq i \leq t: \ \frak n(L_i)= \frak n(L) \} . \end{equation}

\begin{prop} \label{hyper} If there is $j< i(L)$ such that $\frak n(L_j)= \frak n(L)=\frak o_{k_v}$ in the above minimal norm Jordan splitting (\ref{jordan}) over dyadic local fields, then $L_j$ is an orthogonal sum of $2^{-1}A(0,0)$.
\end{prop}

\begin{proof} It follows from \cite[Lemma 1.1]{Xu3}. 
\end{proof}

For any $\xi\in k_v$,  the quadratic defect is defined in \cite[\S 63 A]{OM}
$$ \frak d (\xi) = \bigcap_{\alpha} \alpha \frak o_{k_v} $$ where $\alpha$ runs over all possible expressions $\xi=\eta^2+\alpha$ for some $\eta\in k_v$. Let $\Delta=1+4\rho$ be a unit satisfying $\frak d(\Delta)=4\frak o_{k_v}$ by \cite[63:2]{OM}. 

\begin{prop}\label{ref-min} Suppose that $L= L_1\perp L_2\perp \cdots \perp L_t $ is a minimal norm Jordan splitting such that $\rank (L_1)=2$ and $\frak n(L_1)\neq 2\frak s(L_1)$. 
If  $$(\pi_v)\frak n(L_1)= \frak n(L_2\perp \cdots \perp L_t) \ \ \ \text{ or } \ \ \ \frak n(L_1)= (\pi_v) \frak n(L_2\perp \cdots \perp L_t), $$ then there is a minimal norm Jordan splitting $L= L_1'\perp L_2'\perp \cdots \perp L_t' $ such that $$\frak d(-\pi^{-2s_1}\det(L_1'))\subseteq 4\frak o_{k_v}  \ \ \ \text{with} \ \ \ s_1=\ord_v(\frak s(L_1)) . $$
\end{prop}

\begin{proof} There is $2\leq i\leq t$ such that $\frak n(L_i)=  \frak n(L_2\perp \cdots \perp L_t)$. Since $\frak n(L_1)\neq 2\frak s(L_1)$, one concludes that $\frak n(L_i)\neq 2\frak s(L_i)$. There is a sub-lattice $K$ with $\rank (K)\leq 2$ such that $$L_i=K\perp K^{\perp} \ \ \ \text{ and  } \ \ \ \frak n(L_i)=\frak n(K) \supsetneq \frak n(K^\perp)$$ by \cite[93:18]{OM}. Without loss of generality, we assume that $L_1$ is unimodular and only need to consider the splitting of $L_1\perp K$. Suppose 
\begin{equation} \label{defect} \frak d(-\det(L_1))\supsetneq 4\frak o_{k_v} . \end{equation}
 One can write $ L_1=\frak o_{k_v} x + \frak o_{k_v} y\cong A(\alpha \pi_v^a,  \beta \pi_v^b) $ with  $0\leq a< b$ and $\alpha, \beta \in \frak o_{k_v}^\times$ such that $$\frak d(-\det(L_1))=(\pi_v^{a+b})$$ by \cite[93:17]{OM}. 
 Let $K=\frak o_k u$ or  $K=\frak o_k u+ \frak o v$ such that $Q(u)=\gamma \pi_v^c$ with $\gamma\in \frak o_{k_v}^\times$ and $$\ord_v(\frak n(K))=c \ \ \ \text{and} \ \ \ \ord_v(B(u,v))\leq c < \ord_v(Q(v)) . $$
Since $a+b$ is odd by (\ref{defect}) and \cite[63:2]{OM}, one has $c\leq b$ and $c\equiv b \mod 2$. There are $\xi, \eta \in \frak o_{k_v}^\times$ such that $\beta + \gamma\xi^2= \eta \pi_v^d$ with $d\geq 1$. Let $$L_1'=\frak o_k x+ \frak o_k (y+\xi \pi_v^{\frac{b-c}{2}} u) \subset L_1\perp K .$$ 
Then one gets a new splitting $ L_1\perp K=L_1'\perp K'$ with
$$ K'= \begin{cases} \frak o_{k_v} (u-\sigma x+ \sigma Q(x) y)  \ \ \ & \text{$\rank (K)=1$}  \\
 \frak o_{k_v} (u-\sigma x+ \sigma Q(x) y) +\frak o_{k_v} (v-\tau x +\tau Q(x) y)  \ \ \ & \text{$\rank (K)=2$} \end{cases} $$
where  $$\sigma=\det(L_1)^{-1}\xi \pi_v^{\frac{b-c}{2}} Q(u) \ \ \ \text{ and } \ \ \ \tau= \det(L_1)^{-1}\xi \pi_v^{\frac{b-c}{2}} B(u,v) . $$ 
By a simple computation, one has 
$$ \frak d(-\det(L_1')) \subsetneq \frak d (-\det(L_1))$$ with $\frak n(L_1')=\frak n(L_1)$ and $\frak n(K')=\frak n(K)$. By repeating this argument, one eventually gets 
 $$\frak d(-\det(L_1'))\subseteq 4\frak o_{k_v}$$ as required. 
\end{proof}

\begin{prop} \label{norm} Let $L$ be an $\frak o_{k_v}$-lattice with the minimal norm Jordan splitting (\ref{jordan}) over a dyadic local field $k_v$.  Suppose $\rank (L_{i(L)})=1$ and $\frak n(L)= \frak o_k$. If $L$ is universal, then there is $1\leq j\leq t$ with $j\neq i(L)$ such that $(\pi_v)  \subseteq \frak n (L_j) \subseteq \frak o_{k_v} $. 
\end{prop}

\begin{proof}  Without loss of generality, we assume that $L_{i(L)}=\frak o_{k_v} x$ with $Q(x)=1$. Suppose $$ \frak n(L_k)\subseteq (\pi_v^2) \ \ \ \text{for all $k\neq i(L)$} . $$
Since $L$ is universal, there are $a\in \frak o_{k_v}^\times$ and $$y\in L_1\perp\cdots \perp L_{i(L)-1} \perp L_{i(L)+1} \perp \cdots \perp L_t$$ such that 
\begin{equation} \label{def} 1+\pi_v=Q(ax+y)= a^2+Q(y) .  \end{equation}
Since $Q(y)\in (\pi_v^2)$, a contradiction is derived from (\ref{def}) by \cite[63:5]{OM}. 
\end{proof}

\begin{cor} \label{binary} A binary $\frak o_{k_v}$-lattice $L$  over dyadic local fields is universal if and only if $L=2^{-1}A(0,0)$.
\end{cor}

\begin{proof} Suppose that $L$ has two Jordan components. One can assume that $$ L= \frak o_{k_v} x\perp \frak o_{k_v} y \ \ \ \text{with} \ \ \ Q(x)=1, \ \ Q(y) = \epsilon \pi_v$$ by Prop. \ref{norm}.
Then $\Delta\not\in Q(L)$ by \cite[63:11]{OM}. 

Suppose that $L$ is modular. Then one can assume that $$L\cong \pi_v^{-a}A(\pi_v^a, \epsilon \pi_v^b) \ \ \ \text{with} \ \ \ 0\leq a\leq b \ \ \ \text{and} \ \ \ a \leq \ord_v(2) $$ by \cite[93:17]{OM}. 

If $a< \ord_v(2)$, then $\tau_z\in O(L)$ for any $z\in L$ with $\ord_v(Q(z))\leq 1$. Since $L$ is universal, one concludes that $\theta(O^+(L))=k_v^\times$ by Lemma \ref{cri}. This is impossible by \cite[Prop.B, Prop.C, Prop.D, Prop.E]{Hs} and \cite{Xu1}. 

Otherwise, one has $a=\ord_v(2)$. Then $L\cong 2^{-1}A(0,0)$ or $2^{-1}A(2, 2\rho)$ by \cite[93:11]{OM}. Since $\pi_v\not\in Q(2^{-1}A(2, 2\rho))$ by the principle of domination in \cite{Ri},  one obtains the result as desired. 
\end{proof}

\subsection{Dyadic cases of rank three}
Now we focus on ternary lattices over dyadic local fields. We first need the following lemma.  

\begin{lem}\label{one-dim} Suppose
$$L=\frak o_{k_v} x\perp \frak o_{k_v} y \perp \frak o_{k_v} z \ \ \ \text{with} \ \ \ Q(x)=1, \ \ Q(y)=\epsilon \pi_v, \ \ Q(z)= \delta \pi_v^k  \ \ \text{and}  \ \ 1\leq k \leq 2\ord_v(2)  $$ where $\epsilon, \delta \in \frak o_{k_v}^\times$. Then $\frak o_{k_v}^\times \subseteq Q(L)$ if and only if $k$ is even.
\end{lem}

\begin{proof} For an even integer $k$,  one considers  $$ \frak o_{k_v}^\times = (\frak o_{k_v}^\times \cap Q(\frak o_{k_v} x\perp \frak o_{k_v} y ) )\cup \Delta (\frak o_{k_v}^\times \cap Q(\frak o_{k_v} x\perp \frak o_{k_v} y ) )  $$ by \cite[63:11]{OM}. For any  $$(1+4\rho) (a^2+\epsilon \pi_v b^2) \in \Delta (\frak o_{k_v}^\times \cap Q(\frak o_{k_v} x\perp \frak o_{k_v} y ) ) $$ with $a\in \frak o_{k_v}^\times$ and $b \in \frak o_{k_v}$, there are $\xi\in \frak o_{k_v}^\times$ and $\eta\in (\pi_v)$ such that $\rho\delta^{-1}=\xi^2+\eta$.  By \cite[63:1]{OM}, there is $\gamma\in \frak o_{k_v}^\times$ such that $$\gamma^2=a^2+4\rho\epsilon \pi_v b^2 + 4a^2\delta \eta .$$
Since $$ Q(\gamma x+by +2a\pi_v^{-\frac{k}{2}}\xi z)=\gamma^2+\epsilon \pi_v b^2+ 4 a^2\xi^2 \delta  = a^2+4\rho\epsilon \pi_v b^2 + 4a^2\delta \eta +\epsilon \pi_v b^2+ 4 a^2\xi^2 \delta $$ 
$$=(a^2 +\epsilon \pi_v b^2) +4\rho \epsilon \pi_v b^2 +4a^2\delta (\xi^2+\eta) =(a^2 +\epsilon \pi_v b^2) +4\rho \epsilon \pi_v b^2 +4a^2\rho= (1+4\rho) (a^2+\epsilon \pi_v b^2) $$ with $\gamma x+by +2a\pi_v^{-\frac{k}{2}}\xi z\in L$, one concludes that $\frak o_{k_v}^\times \subseteq Q(L)$. 

For an odd integer $k$, we consider a new quadratic form $Q'$ scaling by $-\epsilon\delta$. Then  $$\frak o_{k_v}^\times \subseteq Q(L) \ \ \ \text{ if and only if } \ \ \  \frak o_{k_v}^\times \subseteq Q'(L) . $$
Suppose $\frak o_{k_v}^\times \subseteq Q'(L)$. Then 
$$L=\frak o_{k_v} x'\perp \frak o_{k_v} y'\perp \frak o_{k_v} z' \ \ \ \text{with} \ \ \ Q'(x')=1, \ \ Q'(y')= \epsilon' \pi_{v} \ \ \text{and} \ \ Q'(z')=\delta'\pi_v^k $$ where $\epsilon', \delta'\in \frak o_{k_v}^\times$ with $-\epsilon'\delta'\in (\frak o_{k_v}^\times)^2$. Since $\Delta \in Q'(L)$, one obtains that 
$$ \Delta \in Q'(\frak o_{k_x} x' \perp (\pi_v^{\frac{k-1}{2}}) y' \perp \frak o_{k_v} z') $$ by \cite[63:5]{OM}. Since 
$(\pi_v^{\frac{k-1}{2}}) y' \perp \frak o_{k_v} z'=\frak o_{k_v} u + \frak o_{k_v} w$ where $$Q'(u)=\gamma \pi_v^k, \ \ \ Q'(w)=0,  \ \ \ B'(u,w)=\pi_v^{k}  \ \ \ \text{with} \ \ \ \gamma\in \frak o_{k_v}^\times,  $$ there are $a\in \frak o_{k_v}^\times$ and $b, c\in \frak o_{k_v}$ such that $\Delta= Q'(ax'+bu+cw)$. By \cite[63:5]{OM}, one concludes that $\ord_v(b)\geq  e-\frac{1}{2}(k-1)$. This implies that $\Delta\in (\frak o_{k_v}^\times)^2$ by \cite[63:1]{OM}, which is a contradiction.   
\end{proof}

\begin{prop}\label{univ-one} A universal $\frak o_{k_v}$-lattice of rank 3 has a Jordan component of rank $\geq 2$. 
\end{prop}
\begin{proof}  Suppose there is a universal $\frak o_{k_v}$-lattice $L$ of rank 3 such that all Jordan components of $L$ are one dimensional.  By Prop.\ref{norm}, one can assume that 
$$L=\frak o_{k_v} x\perp \frak o_{k_v} y \perp \frak o_{k_v} z \ \ \ \text{with} \ \ \ Q(x)=1, \ \ Q(y)=\epsilon \pi_v, \ \ Q(z)= \delta \pi_v^k  \ \ \text{and}  \ \  k \geq 2  $$  where $\epsilon, \delta \in \frak o_{k_v}^\times$. Then $k\leq 2\ord_v(2)$ is even by \cite[63:1]{OM} and Lemma \ref{one-dim}.  Since $$\pi_v \frak o_{k_v}^\times \subset Q(\frak o_{k_v}(\pi_v x)\perp \frak o_{k_v} y \perp \frak o_{k_v} z) , $$ one concludes that 
$$ \frak o_{k_v}^\times \subset Q^{\epsilon^{-1} \pi_v^{-1}}(\frak o_{k_v} y \perp \frak o_{k_v}(\pi_v x) \perp \frak o_{k_v} z ) $$ by scaling $\epsilon^{-1} \pi_v^{-1}$. A contradiction is derived by Lemma \ref{one-dim}. 
\end{proof}

\begin{lem} \label{neg} If $L=(\frak o_{k_v} x+ \frak o_{k_v} y) \perp \frak o_{k_v} z$ such that 
$$ Q(x)=\epsilon \pi_v, \ \ \ B(x, y)=\pi_v^{-a}, \ \ \ Q(y)= -4\rho \epsilon^{-1} \pi_v^{-2a-1} \ \ \ \text{and} \ \ \ Q(z) =1 $$ where $0\leq a< \ord_v(2)$ is an integer and $\epsilon \in \frak o_{k_v}^\times$, then $\Delta \not\in Q(L)$.   
\end{lem}
\begin{proof} Suppose not. There are $\alpha, \beta\in \frak o_{k_v}$ and $\gamma\in \frak o_{k_v}^\times$ such that $$\Delta= Q(\alpha x+\beta y+\gamma z)=\gamma^2+ Q(\alpha x+\beta y) . $$ Since $$\ord_vQ(\alpha x+\beta y)=\min \{ \ord_v(Q(\alpha x)), \ord_v (Q(\beta y)) \} $$ is odd by the principle of domination in \cite{Ri},  a contradiction is derived by \cite[63:1 and 63:5]{OM}. 
\end{proof}

\begin{lem} \label{out} Let  $L=  L_1 \perp \frak o_{k_v} z$ such that 
$$  L_1\cong \pi_v^{-a} A(\pi_v^a, \delta \pi_v^{b-a} )\ \ \ \text{and} \ \ \ Q(z) = \epsilon \pi_v^r $$ in sense of \cite[93:17]{OM}, where $0\leq a< \ord_v(2)$ and $\epsilon, \delta \in \frak o_{k_v}^\times$. If $r\geq 2$, then $L$ is not universal.  
\end{lem}

\begin{proof} Since $a<\ord_v(2)$, one has $b\geq 2a+1$. When $b> 2a+1$, then $\pi_v \not \in Q(L)$. So one only needs to consider that $b=2a+1$. Let $\{x, y\}$ be the corresponding basis of $L_1$. Then 
$$ Q(x)=1, \ \ \ B(x, y)=\pi_v^{-a} \ \ \ \text{and} \ \ \  Q(y)= \delta \pi_v^{b-2a} . $$
Since $L\subset (\frak o_{k_v} x+ \frak o_{k_v} y) \perp \frak o_{k_v} (\pi_v^{-[\frac{r-2}{2}]} z) $, one only needs to prove the result for $r=2, 3$. 

For the case that $k_v L$ is anisotropic,  we claim that $ \Delta \epsilon \pi_v^r \not\in Q(L)$. Suppose not, one has 
$$  \begin{cases}   \Delta \epsilon \pi_v^r   \in Q( (\frak o_{k_v} \pi_v x+ \frak o_{k_v}  \pi_v y)\perp \frak o_{k_v} z )  \ \ \ & r=2 \\
 \Delta \epsilon \pi_v^r  \in Q((\frak o_{k_v} \pi_v^2 x+ \frak o_{k_v}  \pi_v y)\perp \frak o_{k_v} z ) \ \ \ &  r=3 \end{cases} $$
by the principle of domination in \cite{Ri}. A contradiction is derived by \cite[Lemma 1.1]{Xu3}, Prop. \ref{ref-min} and Lemma \ref{neg} with scaling $\epsilon^{-1} \pi_v^{-r}$. 
Therefore we can further assume that $k_v L$ is isotropic. 

\medskip

\underline{Case $r=2$}.  In this case, we claim that $\frak o_{k_v}^\times \not \subseteq Q(L)$. Indeed, 
since $k_v L$ is isotropic, there are $\alpha, \beta$ and $\gamma$ in $\frak o_{k_v}$ with $$\ord_v(\alpha) \cdot \ord_v(\beta) \cdot \ord_v(\gamma)=0 \ \ \ \text{ such that } \ \ \ Q(\alpha x+ \beta y+ \gamma z)= 0 . $$ This implies that $\ord_v(\gamma)=\ord_v(\alpha)-1=0$ and $\ord_v(\beta)\geq 1$ by the principle of domination in \cite{Ri}. Therefore $-\epsilon$ is represented by $L_1$. By scaling $-\epsilon$ if necessary, one can simply assume that $\epsilon =-1$ by \cite[93:17]{OM}. Write $$\det(L_1)=\omega_1 \pi_v^{-2a} \ \ \ \text{with} \ \ \ \omega_1 \in \frak o_{k_v}^\times . $$ Then $\frak d (-\omega_1) =(\pi_v^{2a+1})$ by \cite[63:5]{OM}. By \cite[Lemma 3]{Hs}, there is $\eta\in \frak o_{k_v}^\times$ such that  $$\frak d(\eta)= (\pi_v^{2(\ord_v(2)-a)-1}) \ \ \ \text{and the Hilbert symbol} \ \ \  (\eta, -\omega_1)_v=-1 .$$ Then
$\frak d(\omega_1 \eta)= (\pi_v^s)$ satisfying 
\begin{equation} \label{test-defect}  \begin{cases} s =  2a+1 \ \ \  & \text{ for \ $a<\frac{1}{2}(\ord_v(2)-1)$} \\
s = 2\ord_v(2)-2a-1 \ \ \ & \text{ for \ $a>\frac{1}{2}(\ord_v(2) -1)$} \\
s \geq \ord_v(2)  \ \ \ & \text{ for \  $a=\frac{1}{2}(\ord_v(2) -1)$ .} 
\end{cases} \end{equation}
Suppose that the above claim is not true. Then $\eta\omega_1\in Q(L)$. There are $g\in \frak o_{k_v}^\times$, $c$ and $h$ in $\frak o_{k_v}$ such that 
$$  \omega_1 \eta =g^2+2\pi_v^{-a}gh + \delta \pi_v h^2 - c^2 \pi_v^2$$ 
\begin{equation} \label{test-equality} =[(g+c\pi_v)^2+2\pi_v^{-a}(g+c\pi_v)h+\delta \pi_v h^2]-2\pi_v cg- 2\pi_v^{-a+1}ch-2\pi_v^2 c^2 .   \end{equation}

When $\ord_v(h) \geq \ord_v(2)-a-1$, then $$ s \geq \min \{\ord_v(2)-a+\ord_v(h),  \ord_v(2) +1 \} \geq \min \{ 2(\ord_v(2)-a)-1, \ord_v(2) +1 \} $$ by (\ref{test-equality}). A contradiction is derived for $a<\frac{1}{2}(\ord_v(2)-1)$ by (\ref{test-defect}). For $a\geq \frac{1}{2}(\ord_v(2)-1)$, one obtains that 
$$\ord_v(2\pi_v cg+ 2\pi_v^{-a+1}ch+2\pi_v^2 c^2) \geq \min \{\ord_v(2)-a+1+ \ord_v(h),  \ord_v(2) +1 \} $$
$$\geq \min \{ 2(\ord_v(2)-a), \ord_v(2) +1 \} = 2(\ord_v(2)-a) . $$
By \cite[Remark 1]{Xu2}, one concludes that the Hilbert symbol
$$(1-\xi^{-1}(2\pi_v cg+ 2\pi_v^{-a+1}ch+2\pi_v^2 c^2),  \ -\omega_1)_v =1 $$
with $$ \xi=(g+c\pi_v)^2+2\pi_v^{-a}(g+c\pi_v)h+\delta \pi_v h^2 . $$ Since $\xi\in Q(k_v L_1)$, one has $(\xi, -\omega_1)_v=1$.  Therefore 
$$1= (\xi, -\omega_1)_v\cdot ([1-\xi^{-1}(2\pi_v cg+ 2\pi_v^{-a+1}ch+2\pi_v^2 c^2)], -\omega_1)_v=(\omega_1\eta, -\omega_1)_v =(\eta, -\omega_1)_v$$ by (\ref{test-equality}), which contradicts to the choice of $\eta$. 

When $\ord_v(h) < \ord_v(2)-a-1$, then $$  2\ord_v(h) +1 <  \min \{ \ord_v(2)-a+ \ord_v(h) , \ord_v(2) +1 \} $$ for $a>\frac{1}{2}(\ord_v(2)-1)$ or $\ord_v(h) < a$. This implies that 
$$s=2\ord_v(h)+1 < \begin{cases}  2 (\ord_v(2) -a) -1 \ \ \ & \text{ for $a>\frac{1}{2}(\ord_v(2)-1)$} \\ 
2a + 1 \ \ \ & \text{ for $\ord_v(h) < a $} \end{cases} $$ by (\ref{test-equality}) and \cite[63:5]{OM}. A contradiction is derived by (\ref{test-defect}) except $$a\leq \ord_v(h) < \ord_v(2)-a-1 . $$ 

For $a< \ord_v(h) < \ord_v(2)-a-1$, one has 
$$\ord_v(2\pi_v^{-a}(g+c\pi_v)h+\delta \pi_v h^2-2\pi_v cg- 2\pi_v^{-a+1}ch-2\pi_v^2 c^2) \geq 2\ord_v(h)+1 >2a+1.$$ Then 
the Hilbert symbol $$(\eta, \ 1+2\pi_v^{-a}(g+c\pi_v)h+\delta \pi_v h^2-2\pi_v cg- 2\pi_v^{-a+1}ch-2\pi_v^2 c^2)_v=1 $$ by \cite[Remark 1]{Xu2}. This implies that 
$$1= (\eta, \omega_1 \eta)_v = (\eta, -\omega_1)_v$$ by (\ref{test-equality}), which contradicts to the choice of $\eta$. 

For $\ord_v(h)=a<\frac{1}{2}(\ord_v(2)-1)$, one obtains that $ \omega_1 \eta  \in Q(M) $ with $$M= A(1, \delta\pi_v^{2a+1}) \perp < -\pi_v^2> = <1> \perp <\omega_1> \perp <-\pi_v^2>$$ by (\ref{test-equality}). There are $u$, $v$ and $l$ in $\frak o_{k_v}$ such that 
\begin{equation} \label{mini-latt}  \omega_1 \eta = u^2 + \omega_1 v^2 - \pi_v^2 l^2 = (u+\pi_vl)^2 +\omega_1 v^2 - 2 \pi_v u l -2 \pi_v^2 l^2 . \end{equation}  Then 
$$ \eta= (u+\pi_v l)^2 \omega_1^{-1} + v^2 -2\pi_v l \omega_1^{-1} (u+\pi_v l). $$
Since $$\frak d (\omega_1^{-1}) =\frak d (-\omega_1) =(\pi_v^{2a+1}) \ \ \ \text{and} \ \ \ \frak d(\eta)= (\pi_v^{2(\ord_v(2)-a)-1}) \ \ \ \text{with} \ \ \ 2a+1 < \ord_v(2) , $$ one concludes that $$v\in \frak o_{k_v}^\times \ \ \ \text{and} \ \ \ \ord_v(u+\pi_v l) \geq \ord_v(2) -2a -1 $$ by  \cite[63:5]{OM}. This implies that 
$$\ord_v(2\pi_v l \omega_1^{-1} (u+\pi_v l)) \geq 2 \ord_v(2) -2a \ \ \ \text{and} \ \ \ (-\omega_1, \ 1-\lambda^{-1}(2\pi_v l \omega_1^{-1} (u+\pi_v l)) )_v=1$$ 
with $\lambda= (u+\pi_v l)^2 \omega_1^{-1} + v^2$ by \cite[Remark 1]{Xu2}. Since $$(\lambda, -\omega_1)_v=1 \ \ \ \text{ and } \ \ \ \eta= \lambda (1-\lambda^{-1}(2\pi_v l \omega_1^{-1} (u+\pi_v l)) , $$ one obtains that $(-\omega_1, \eta)_v=1$ which contradicts to the choice of $\eta$. 

\medskip

\underline{Case $r=3$}.  In this case, we claim that $\pi_v \frak o_{k_v}^\times \not \subseteq Q(L)$. Suppose not, one has 
$$ \pi_v \frak o_{k_v}^\times  \subseteq  Q( (\frak o_{k_v} \pi_v x+ \frak o_{k_v} y)\perp \frak o_{k_v} z )  $$ and reduces to the above case by scaling $\delta^{-1} \pi_v^{-1}$. A contradiction is derived.
\end{proof}

\begin{lem}\label{out1}  If  $L=  L_1 \perp \frak o_{k_v} z$ such that 
$$  L_1\cong 2^{-1} A(2, 2\rho)\ \ \ \text{and} \ \ \ Q(z) = \epsilon \pi_v^r  \ \ \ \text{with} \ \ \ r\geq 2,$$
 then $\pi_v\not\in Q(L)$. 
\end{lem}

\begin{proof} It follows from the principle of domination in \cite{Ri}.
\end{proof}

\begin{lem} \label{pos}  Let  $L= (\frak o_{k_v} x + \frak o_{k_v} y) \perp \frak o_{k_v} z$ such that
$$   \underline{Case \ 1.}  \ \ \ 
 Q(x)=1, \ \ \ B(x, y)=\pi_v^{-a}, \ \ \ Q(y)=0 \ \ \ \text{and} \ \ \ Q(z) =\epsilon \pi_v $$
$$ \underline{Case \ 2.} \ \ \ 
 Q(x)=\epsilon \pi_v, \ \ \ B(x, y)=\pi_v^{-a}, \ \ \ Q(y)= 0 \ \ \ \text{and} \ \ \ Q(z) =1 $$ where $0\leq a< \ord_v(2)$ is an integer and $\epsilon \in \frak o_{k_v}^\times$.   

If $L$ is one of the above lattices, then $\frak o_{k_v}^\times \subseteq Q(L)$.  
\end{lem}

\begin{proof} Since $$ \frak o_{k_v}^\times = (\frak o_{k_v}^\times \cap Q(\frak o_{k_v} x\perp \frak o_{k_v} z ) )\cup \Delta (\frak o_{k_v}^\times \cap Q(\frak o_{k_v} x\perp \frak o_{k_v} z ) )  $$ by \cite[63:11]{OM}, one only needs to show that  $$\Delta (\frak o_{k_v}^\times \cap Q(\frak o_{k_v} x\perp \frak o_{k_v} z )) \subset Q(L) . $$ 
Indeed, for any $(1+4\rho) Q(\beta x+ \gamma z)\in \frak o_{k_v}^\times $ with $\beta, \gamma \in \frak o_{k_v}$,  one can take $\xi \in \frak o_{k_v}^\times$ such that 
$$ \xi^2 = \begin{cases} \beta^2 + 4\rho \pi_v \epsilon \gamma^2 \ \ \ \ & \text{for Case 1} \\
\gamma^2 + 4\rho \pi_v \epsilon \beta^2  \ \ \ \ & \text{for Case 2} \end{cases} $$
by \cite[63:1]{OM}. For Case 2, one can assume that $ord_v(\beta) < ord_v(2)$. Otherwise, one has $$Q(\beta x+ \gamma z) \in (\frak o_{k_v}^\times)^2 $$ by  \cite[63:1]{OM}.
Since $$2x+(\rho-\epsilon \pi_v) \pi_v^a y +z \in L \ \ \ \text{with} \ \ \ Q(2x+(\rho-\epsilon \pi_v) \pi_v^a y +z)= 1+4\rho , $$ one concludes that $(1+4\rho) Q(\beta x+ \gamma z) \in Q(L)$
as desired. 

Let 
$$ \eta= \begin{cases} 2\rho \pi_v^a \beta^2 \xi^{-1} \ \ \ & \text{for Case 1} \\
2\rho \pi_v^a \gamma^2 \beta^{-1} \ \ \ & \text{for Case 2} \end{cases} $$ in $\frak o_{k_v}$ and
 $$ w= \begin{cases} \xi x+ \eta y + \gamma z  \ \ \ & \text{for Case 1} \\
\beta x+ \eta y + \xi z \ \ \ & \text{for Case 2} \end{cases} $$
in $L$. Then $(1+4\rho) Q(\beta x+ \gamma z)=Q(w) \in Q(L)$ as desired. 
\end{proof}

\begin{prop} \label{univ-bi} Let $L$ be an $\frak o_{k_v}$-lattice of rank 3 with the minimal norm Jordan splitting $L=L_1\perp L_2$. Then $L$ is universal if and only if $\frak n(L)=\frak o_{k_v}$ and $\rank (L_1)=2$ and one of the conditions holds.

1) $L_1\cong 2^{-1}A(0,0)$.

2) $(\pi_v)\frak n(L_1)= \frak n(L_2)$ and $k_vL$ is isotropic.

3) $\frak n(L_1)= (\pi_v) \frak n(L_2)$ and $k_vL$ is isotropic. 
\end{prop}

\begin{proof} Sufficiency. For Case 1), since $2^{-1}A(0,0)$ is universal, one obtains that $L$ is universal. 

For Case 2), one can assume that $$\frak d(-\det(L_1)\pi^{-2s_1})\subseteq 4\frak o_{k_v}$$ by Prop. \ref{ref-min}, where $s_1=\ord_v(\frak s(L_1))$. Since $k_v L$ is isotropic, one concludes that $$\frak d(-\det(L_1)\pi^{-2s_1})= 0$$ by the principle of domination in \cite{Ri}.  
Therefore $\frak o_{k_v}^\times \subseteq Q(L)$ by Lemma \ref{pos} Case 1. 

Write $L_1=\frak o_{k_v}x+ \frak o_{k_v} y$ with $$Q(x)\in \frak o_{k_v}^\times, \ \ \ B(x, y)=\pi_v^{s_1} \ \ \ \text{and} \ \ \  Q(y)=0 . $$
Then $$ \pi_v \frak o_{k_v}^\times \subset Q(((\pi_v) x+ \frak o_{k_v} y) \perp L_2) \subset Q(L) $$
by Lemma \ref{pos} Case 2 with scaling $\pi_v$. Therefore $L$ is universal by Lemma \ref{cri}.

Case 3) follows from the same argument as that of Case 2). 

\bigskip
 
Necessity.  Since $Q(L)=\frak o_{k_v}$, one has $\frak n(L)=\frak o_{k_v}$. Suppose $\rank (L_1)=1$. This implies  $$\frak n(L_1)=\frak s(L_1)=\frak n(L)=\frak o_k  \ \ \ \text{and} \ \ \ \frak s(L_2)=\frak n(L_2) =(\pi_v)  $$ by Prop. \ref{norm}. Therefore $\frak o_{k_v}^\times \nsubseteq Q(L)$ by Lemma \ref{one-dim}. A contradiction is derived. 

If $i(L)=1$ and $L_1\not \cong 2^{-1}A(0,0)$,  one obtains $(\pi_v)\frak n(L_1)= \frak n(L_2)$ by Lemma \ref{out} and Lemma \ref{out1}. Suppose that $k_v L$ is not isotropic. 
By Prop.\ref{ref-min}, one can assume  $$\frak d(-\det(L_1)\pi_v^{-2s_1})=4\frak o_{k_v}$$ with $s_1=\ord_v(\frak s(L_1))$. 
Then $\pi_v\frak o_{k_v}^\times \nsubseteq Q(L)$ by Lemma \ref{neg} with scaling $\pi_v^{-1}$ and a contradiction is derived. 

If $i(L)=2$ and $L_1\not \cong 2^{-1}A(0,0)$, one obtains $\frak n(L_1)= (\pi_v) \frak n(L_2)$ by Prop.\ref{hyper} and Prop.\ref{norm}. Suppose that $k_v L$ is not isotropic. Then $\frak o_{k_v}^\times \nsubseteq Q(L)$ by Prop.\ref{ref-min} and Lemma \ref{neg} and a contradiction is derived. 
\end{proof}

Recall the norm group $\frak g(L)$ and the weight $\frak w(L)$ of $\frak o_{k_v}$-lattice $L$ is defined as follows  $$ \frak g(L)= Q(L) + 2\frak s(L) \ \ \ \text{and} \ \ \ \frak w(L)= (\pi_v )\frak m(L)+ 2\frak s(L)$$ where $\frak m(L)$ is the largest fractional ideal contained in $\frak g(L)$ by \cite[\S 93 A]{OM}. 

\begin{prop}\label{univ-ter} A unimodular $\frak o_{k_v}$-lattice $L$ of rank 3 over dyadic local fields is universal if and only if $\frak w(L)=(\pi_v)$ and $k_vL$ is isotropic. 
\end{prop}
\begin{proof}  Sufficiency. Since $k_vL$ is isotropic, one has $$ L= \frak o_{k_v} x \perp (\frak o_{k_v} y + \frak o_{k_v} z)   \ \ \text{with} \ \  Q(x)\in \frak o_{k_v}^{\times} ,  \ \ \ord_v(Q(y))=1,  \ \ B(y, z)=1  \ \ \text{and}   \ \ Q(z) =0  $$ by \cite[93:18 (iv)]{OM} and the principle of domination in \cite{Ri}. Then $\frak o_{k_v}^\times \subseteq Q(L)$ by Lemma \ref{pos} Case 2.  Since 
$$ \pi_v \frak o_{k_v}^\times \subset Q((\pi_v) x \perp (\frak o_{k_v} y + \frak o_{k_v} z) )\subset Q(L) $$ by Lemma \ref{pos} Case 1 with scaling $\pi_v$, one concludes that $L$ is universal by Lemma \ref{cri}. 

\bigskip

Necessity.  Suppose $\frak w(L)\subseteq (\pi_v^2)$. Since $ L= \frak o_{k_v} x \perp (\frak o_{k_v} y + \frak o_{k_v} z) $  with 
$$ Q(x)\in \frak o_{k_v}^{\times} ,  \ \ \ord_vQ(y)=\ord_v\frak w(L),  \ \ B(y, z)=1  \ \ \text{and}   \ \  \ord_v(Q(z)) \geq 2\ord_v2 -\ord_vQ(y)  $$ by \cite[93:18 (iv)]{OM}, one concludes
 $Q(x) (1+\pi_v)\not\in Q(L)$ by \cite[63:5]{OM}. A contradiction is derived.
 
 Suppose $k_vL$ is anisotropic. Then $ L= \frak o_{k_v} x \perp (\frak o_{k_v} y + \frak o_{k_v} z) $  with 
$$ Q(x)\in \frak o_{k_v}^{\times} ,  \ \ \ord_vQ(y)=1,  \ \ B(y, z)=1  \ \ \text{and}   \ \  \frak d(1-Q(y)Q(z))=4\frak o_{k_v}  $$ by \cite[93:18 (iv)]{OM}. Therefore $Q(x) \Delta \not\in Q(L)$ by \cite[63:1, 63:5]{OM} and the principle of domination in \cite{Ri}. A contradiction is derived. 
\end{proof}

\section{Indefinite universal quadratic forms over number fields}

In this section, we will apply our computations in \S \ref{local} to study indefinite universal quadratic forms over a number field $k$.

\begin{defi}  Let $L$ be an $\frak o_k$-lattice.

(1) We say that $L$ is (globally) universal if $Q(L)=\frak o_k$. 

(2) We say that $L$ is locally universal if $L_v$ is universal for all $v\in \Omega_k\setminus \infty_k$ and $Q(k_v L)=k_v$ for all $v\in \infty_k$. 

\end{defi}

Recall that two $\frak o_k$-lattices $M$ and $N$ in the same quadratic space $V=kM=kN$ are in same class if there is $\sigma\in O(V)$ such that $M=\sigma N$.  This class is denoted by $cls(M)$. Moreover, we say $M$ and $N$ are in the same proper class if there is $\tau\in O^+(V)$ such that $M=\tau N$. This sub-class is denoted by $cls^+(M)$. 

Let $L$ be an $\frak o_k$-lattice on a quadratic space $(V, Q)$ and $O_\Bbb A(V)$ be the adelic group of $O(V)$. One can define $gen(L)$ to be the orbit of $L$ under the action of $
O_\Bbb A(V)$. It is well known that the number of proper classes in $gen(L)$ is finite by \cite[103:4]{OM}. 

\begin{prop}\label{one-class} Suppose that $|\Pic(\frak o_k)|$ is odd and $L$ is an $\frak o_k$-lattice with $\rank (L)\geq 3$. If $L$ is locally universal, then $gen(L)$ contains only a single proper class. In particular, $L$ is universal.
\end{prop}
\begin{proof} By \cite[101:8; 102:7;  104:5]{OM}, the number $h^+(L)$ of proper classes in $gen(L)$ is given by $$h^+(L)= [\Bbb I_k: k^\times (\prod_{v\in \infty_k} k_v^\times \times \prod_{v\not\in \infty_k }\theta(O^+(L_v)))]
$$ where $\Bbb I_k$ is the idelic group of $k$, $O^+(L_v)=\{\sigma \in O^+(V_v): \sigma (L_v) =L_v \}$ and $\theta$ is the spinor norm map. Since $$\theta(O^+(L_v))\supseteq \frak o_{k_v}^\times (k_v^\times)^2$$ by Lemma \ref{cri}, one has 
$$ h^+(L) \mid  [ \Bbb I_k: k^\times ((\prod_{v\in \infty_k} k_v^\times)\times (\prod_{v\not\in \infty_k} \frak o_{k_v}^\times)) \Bbb I_k^2 ] .$$ Since 
$$ |\Pic(\frak o_k)|=[\Bbb I_k : k^\times ((\prod_{v\in \infty_k} k_v^\times)\times (\prod_{v\not\in \infty_k} \frak o_{k_v}^\times)) ]$$ 
by \cite[33:14]{OM}, one concludes that $h^+(L)$ divides 2-part of $|\Pic(\frak o_k)|$. Since $|\Pic(\frak o_k)|$ is odd, one obtains $h^+(L)=1$ as desired. This in particular implies that $L$ is universal.
\end{proof}

As \cite[Theorem 3 and Theorem 4]{Ro} over $\Bbb Z$, one can apply Prop.\ref{one-class} to determine all universal ternary quadratic forms over the number fields of which the class numbers are odd by Prop.\ref{nondyadic}, Prop.\ref{univ-one}, Prop.\ref{univ-bi} and Prop.\ref{univ-ter}.  
On the other hand, there do exist universal quadratic forms such that the genera contain more than one (proper) classes. 

\begin{exam} Let $k=\Bbb Q(\sqrt{-p q_1\cdots q_s})$ where $p$ is a primes with $p\equiv 3 \mod 8$  and $q_i$ is a prime with $q_i\equiv 1 \mod 8$ for $1\leq i\leq s$. Consider $f=x^2+y^2+z^2$ over $\frak o_k$. Then $f$ is universal over $\frak o_k$ and $\text{gen}(f)$ contains $2^s$ (proper) classes. 
\end{exam}

\begin{proof} By \cite[Theorem]{EH}, the quadratic form $f$ is universal over $\frak o_k$. By \cite[92:5]{OM} and \cite[Prop.A]{Hs}, the number $h(f)$ of (proper) classes in $gen(f)$ is given by 
$$ h(f)= [ \Bbb I_k: k^\times ((\prod_{v\in \{2, \infty_k\}} k_v^\times)\times (\prod_{v\not\in \{2, \infty_k\}} \frak o_{k_v}^\times)) \Bbb I_k^2 ] $$ where $\Bbb I_k$ is the idelic group of $k$ as Prop.\ref{one-class}. Since $(i_v)_{v\in \Omega_k}=2\cdot (j_v)_{v\in \Omega_k}$ where 
$$ i_v= \begin{cases} 1 \ \ \ &\text{$v\neq 2$} \\  
2 \ \ \ & \text{$v=2$} \end{cases}  \ \ \ \text{and} \ \ \  j_v =\begin{cases} 2^{-1} \ \ \ & \text{$v\neq 2$} \\
1 \ \ \ & \text{$v=2$} \end{cases} ,$$ one concludes that 
$$ k^\times ((\prod_{v\in \{2, \infty_k\}} k_v^\times)\times (\prod_{v\not\in \{2, \infty_k\}} \frak o_{k_v}^\times))= k^\times ((\prod_{v\in  \infty_k} k_v^\times)\times (\prod_{v\not\in \infty_k} \frak o_{k_v}^\times)) . $$ Therefore 
$$ h(f)= [\Pic(\frak o_k) :  \Pic(\frak o_k)^2] = 2^s $$  by \cite[33:14]{OM} and \cite[Theorem 6.1]{Cox}. 
\end{proof}

By strong approximation for spin groups, every integral quadratic forms with more than three variables which is locally universal is globally universal. However, this is not true any more for integral quadratic forms with two or three variables. Indeed, Estes and Hsia \cite[Theorem]{EH} determined all imaginary quadratic fields such that $x^2+y^2+z^2$ is locally universal but not globally. 

Let $L$ be a ternary $\frak o_k$-lattice such that $k_vL$ is indefinite for some $v\in \infty_k$. An integer $a\in \frak o_k$ is called exceptional for $\text{gen}(L)$ if $a$ is represented by $\text{gen}(L)$ but is not represented by every lattice in $\text{gen}(L)$. 
A purely local criterion for determining exceptional integers was given by Schulze-Pillot in \cite{SP} and generalized by Hsia, Shao and Xu in \cite{HSX}. 

\begin{thm}\label{SP}  (Schulze-Pillot) Let $L$ be a ternary $\frak o_k$-lattice such that $k_vL$ is indefinite for some $v\in \infty_k$. An integer $a\in \frak o_k$ is exceptional for $\text{gen} (L)$ if and only if $-a \cdot \det(L)\not\in (k^\times)^2$ and 
$$ \theta_v (X(L_v, a) ) = N_{k_v(\sqrt{-a \cdot \det(L)})/k_v}(k_v(\sqrt{-a \cdot \det(L)})^\times) \ \ \ \text{for all $v\in \Omega_k$} $$ where $\theta_v$ is the spinor norm map and 
$$ X(L_v, a) = \{ \sigma\in O^+(k_vL): \ x\in \sigma (L_v) \} $$ with $x\in L_v$ satisfying $Q(x)=a$. 
\end{thm} 

The explicit computation of $X(L_v, a)$ over non-dyadic and unramified dyadic fields is given in \cite{SP} and over general dyadic fields in \cite{Xu4}. 

It is natural to ask what kind of number fields admit an integral quadratic form with two or three variables which are locally universal but not globally. We give a complete answer to this question by applying Theorem \ref{SP}. 

\begin{thm}\label{infinite} Let $k$ be a number field. 

(1)  There is a binary $\frak o_k$-lattice which is locally universal but not globally if and only if $|\Pic(\frak o_k)| > 1$. In this case, each binary $\frak o_k$-lattice $L$ which is locally universal but not globally is given by 
$$  L= \frak a x+ \frak a^{-1} y \ \ \ \text{with} \ \ \ Q(x)=Q(y)=0 \ \ \ \text{and} \ \ \ B(x,y)=\frac{1}{2} $$
where $\frak a$ is a non-principal fractional ideal of $k$. Moreover, there is one to one correspondence between the set of proper classes of such binary $\frak o_k$-lattices with the set of non-trivial classes in $\Pic(\frak o_k)$ by sending $L$ to $\overline{\frak a}\in \Pic(\frak o_k)$. 

(2)  There is a ternary $\frak o_k$-lattice which is locally universal but not globally if and only if $|\Pic(\frak o_k)|$ is even. In this case, there are infinitely many classes of such ternary free $\frak o_k$-lattices. 
\end{thm}

\begin{proof}  For (1), we consider a binary $\frak o_k$-lattice $L$ which is locally universal. Then $kL$ is isotropic and $L$ is $\frac{1}{2}\frak o_{k}$-modular by Prop.\ref{nondyadic} and Cor.\ref{binary}. There is a fractional ideal $\frak a$ of $\frak o_k$ such that $$ L= \frak a x+ \frak a^{-1} y \ \ \ \text{ with } \ \ \ Q(x)=Q(y)=0 \ \ \ \text{ and } \ \ \ B(x,y)=\frac{1}{2}$$ by \cite[82:21a]{OM}. 

If $\frak a$ is principal, then $L$ is universal.

If $\frak a$ is not principal, we claim that $L$ is not universal. Indeed, suppose not, then $1\in Q(L)$. There are $a\in \frak a$ and $b\in \frak a^{-1}$ such that $Q(ax+by)=ab=1$. Namely, both $a \frak a^{-1}$ and $b \frak a$ are contained in $\frak o_k$. By the unique factorization theorem for Dedekind domain \cite[22:7]{OM}, one concludes that $\frak a=(a)$ and $\frak a^{-1}=(b)$ which is a contradiction.  

Since any $\sigma\in O^+(V)$ is given by $$\sigma(x)=\alpha x; \ \ \ \sigma(y)=\alpha^{-1}y $$ for $\alpha \in k^\times$, this implies that the above correspondence is well-defined and bijective.

\bigskip

For (2),  the necessity follows from Prop.\ref{one-class}. Now we prove the sufficiency. Since $|\Pic(\frak o_k)|$ is even,  there is an unramified quadratic extension $K=k(\sqrt{a})$ with $a\in \frak o_k$ by \cite[Chapter IV, (7.1) Theorem]{Neu}.  Let 
$$L= (\frak o_k x + \frak o_k y) \perp \frak o_k z \ \ \ \text{with} \ \ \ Q(x)=Q(y)=0, \ \ \ B(x,y)=\frac{1}{2} \ \ \ \text{and} \ \ \ Q(z)= 4a . $$
Then every lattice in $\text{gen} (L)$ is locally universal by Prop.\ref{nondyadic} and Corollary \ref{binary}. We claim that $1$ is an exceptional integer for $\text{gen} (L)$. Indeed, since $K=k(\sqrt{a})$ is unramified, one concludes that $\ord_v(a)$ is even for all non-archimedean primes and the quadratic defect $$\frak d(\pi_v^{-\ord_v(a)} a) \subseteq (\pi_v^{2\ord_v(2)}) $$ for all dyadic primes $v$. Then one can verify 
$$ \theta_v (X(L_v, 1)) = N_{K_{\frak P}/k_v}(K_\frak P^\times) \ \ \ \text{for all $v\in \Omega_k$}  $$ where $\frak P$ is a prime of $K$ above $v$ by \cite[Satz 3]{SP} and \cite[Theorem 2.1]{Xu4}. 
By Theorem \ref{SP}, there is an $\frak o_k$-lattice $M\in \text{gen} (L)$ such that $M$ does not represent $1$. Namely, $M$ is locally universal but not globally universal as desired. 

One can even construct such $M$ explicitly. Indeed, choose an ideal $\frak a$ of $\frak o_k$ such that the image of $\frak a$ in $\Gal (K/k)$ is not trivial under the Artin map by \cite[Chapter IV, (8.2) Theorem]{Neu1}. Then we take
$$M=(\frak a x+ \frak a^{-1} y) \perp \frak o_k z \in gen (L) $$ which is a free $\frak o_k$-module by \cite[81:5 and 81:8]{OM}. By \cite[Theorem 4.1]{HSX}, one concludes that $1\not\in Q(M)$. Based on $M$, we can construct infinitely many classes of free $\frak o_k$-lattices which are locally universal but not globally.  Indeed, by Tchebotarev's density theorem (see \cite[Chapter V, (6.4) Theorem]{Neu1}), there are infinitely many principal prime ideals $\frak p_i$ of $\frak o_k$ for $i=1, \cdots$. Let
$$ M_i = (\frak a x+ \frak a^{-1} y) \perp \frak p_i z \subset M$$ for $i=1, \cdots$. Then $M_i$ is locally universal by Prop.\ref{nondyadic} and Prop.\ref{univ-bi} and free by \cite[81:5 and 81:8]{OM} for $i=1,\cdots$. Since $1\not\in Q(M)$, one obtains that $1\not\in Q(M_i)$ for $i=1,\cdots$. Since $\det(M_i)\neq \det(M_j)$ for $i\neq j$,  one has $\text{cls} (M_i)\neq \text{cls} (M_j)$ for $i\neq j$. The proof is complete.
\end{proof}

\begin{cor} \label{binary-glo} A locally universal binary $\frak o_k$-lattice $L$ is globally universal if and only if $L$ represents 1. In this case, 
$$ L= \frak o_k x+ \frak o_k y \ \ \ \text{with} \ \ \ Q(x)=Q(y)=0 \ \ \ \text{and} \ \ \ B(x,y)=\frac{1}{2} .$$  
\end{cor}
\begin{proof} It follows from the proof of Theorem \ref{infinite} (1). 
\end{proof}

\begin{rem} For Theorem \ref{infinite} (1), one also has one to one correspondence 
$$ \text{ \{ classes of binary locally universal $\frak o_k$-lattices \} }  \longleftrightarrow  \  \Pic(\frak o_k)/\mu_2 $$ where $\mu_2$ acts on $\Pic(\frak o_k)$ by $(-1)\circ \frak a= \frak a^{-1}$ for any $\frak a\in \Pic(\frak o_k)$. The unique globally universal binary lattice corresponds the trivial element in $\Pic(\frak o_k)/\mu_2$. This is because $$O(V)=O^+(V)\cup \tau O^+(V)$$ where $V=Fx+Fy$ with $Q(x)=Q(y)=0, B(x,y)=\frac{1}{2}$ and $\tau(x)=y, \ \tau(y)=x$. 
\end{rem}

\begin{exam} Since every locally universal binary $\frak o_k$-lattice $L$ is free by Theorem \ref{infinite} (1) and \cite[81:5 and 81:8]{OM}, one can express such binary lattices in term of homogeneous polynomials of degree 2 with two variables. 

Let $k= \Bbb Q(\sqrt{-5})$. The class number of $k$ is 2 by the class number formula (see \cite[Chapter VII, (5.11) Corollary, (2.9) Theorem and (10.5) Corollary]{Neu1}). By using Minkowski bound (see \cite[Chapter I, \S 6, Exercise 3]{Neu1}), one has the ideal $\frak a= (2, 1+\sqrt{-5})$ is an ideal of $\frak o_k$ above $2$ which is not principal. 

Consider a binary $\frak o_k$-lattice $$L= \frak a \xi_1+ \frak a^{-1} \xi_2 \ \ \ \text{with} \ \ \ Q(\xi_1)=Q(\xi_2)=0 \ \ \ \text{and} \ \ \ B(\xi_1, \xi_2)=\frac{1}{2} $$ which is universal by Prop.3.2. Since $$\frak a^{-1}=(\frac{1-\sqrt{-5}}{2}, 2) \ \ \ \text{ and } \ \ \ (1+\sqrt{-5})\frak a^{-1} + \frak a= \frak o_k , $$ one can make the following substitution 
$$ \begin{cases} \eta_1=(1+\sqrt{-5})\xi_1+\xi_2 \\ \eta_2= 2\xi_1+\frac{1}{2}(1-\sqrt{-5}) \xi_2  .\end{cases} $$ Then $L=\frak o_k \eta_1+ \frak o_k \eta_2 $ by \cite[81:8]{OM} and the corresponding quadratic form is 
$$ (1+\sqrt{-5})x^2 +5 xy + (1-\sqrt{-5}) y^2 $$ which is unique class of locally universal but not globally binary quadratic form. 

By Theorem \ref{infinite} (2), there are infinitely many ternary free $\frak o_k$-lattices which are locally universal but not globally. Since the Hilbert class field of $k$ is given by $k(\sqrt{-1})$ by \cite[Theorem 6.1]{Cox}, one can exhibit the following infinitely many ternary integral quadratic forms which are locally universal but not globally
$$ (1+\sqrt{-5})x^2 +5 xy + (1-\sqrt{-5}) y^2-4p^2 z^2 $$ where $p$'s are the primes such that $(\frac{-5}{p})=-1$ and $p\equiv 1 \mod 4$ by the proof of Theorem \ref{infinite} (2).  
\end{exam}

\bigskip

For ternary $\frak o_k$-lattices, we extend and refine  \cite[Theorem 22]{Di} (see also \cite[Theorem 5]{Ro}) to general number fields. 

\begin{prop} \label{ternary-glo} Let $L$ be an $\frak o_k$-lattice of rank 3. If $L$ is locally universal, then $kL$ is isotropic. 
\end{prop}
\begin{proof}  By Prop.\ref{nondyadic}, Prop.\ref{univ-one}, Prop.\ref{univ-bi} and Prop.\ref{univ-ter}, one obtains that $k_vL_v$ is isotropic. Since $k_v L$ is also isotropic for all archimedean primes $v$, one concludes that $kL$ is isotropic by Hasse-Minkowski Theorem (see \cite[66:1]{OM}). 
\end{proof}

In \cite[Theorem 6]{Ro}, Ross further proved that every ternary quadratic form with non-trivial zero over $\Bbb Q$ and square-free discriminant over $\Bbb Z$ is universal. Such a result does not hold over  general number fields. 

\begin{exam} \label{ross-glo} Let $k=\Bbb Q(\sqrt{-pq})$ where $p\equiv 7 \mod 8$ and $q\equiv 5 \mod 8$ are primes with $(\frac{q}{p})=1$.  The quadratic form $f=x^2+y^2+z^2$ is locally universal over $\frak o_k$ with discriminant $1$ by \cite[Prop.2.2]{JWX}. Then $f$ is isotropic over $k_v$ for all $v\in \Omega_k$ by Prop.\ref{nondyadic}, Prop.\ref{univ-one}, Prop.\ref{univ-bi} and Prop.\ref{univ-ter}. Therefore $f$ is isotropic over $k$ by Hasse-Minkowski Theorem (see \cite[66:1]{OM}). However, $f$ is not universal over $\frak o_k$ by \cite[Theorem]{EH}. 
\end{exam}

Contrary to the positive definite case, one has the following result.

\begin{prop} \label{counter-glo} For any positive integer $N$, there are infinitely many classes of integral quadratic forms over $\Bbb Z$ which represents all integers between $-N$ and $N$ but are not universal. 
\end{prop}

\begin{proof}  Let $p$ and $q$ be primes with $p>N$ and $q>N$ such that $p\equiv q\equiv 3 \mod 4$. Consider indefinite quadratic form $f=x^2+y^2-pq z^2$.  

First of all, we claim that $f$ is not universal over $\Bbb Z$. Indeed, since $p\equiv q\equiv 3 \mod 4$, one obtains that $f$ is not locally universal at $v=p$ and $q$ by Prop.\ref{nondyadic}. In this case, the Hasse symbol of $f$ is not equal to $(\frac{-1,-1}{\Bbb Q_v})$ for $v=p$ and $q$ by \cite[58:6]{OM}.

On the other hand, we have that $f$ is locally universal at all odd primes $v\neq p, q$ by Prop.\ref{nondyadic}. This implies that the Hasse symbol of $f$ is equal to $(\frac{-1,-1}{\Bbb Q_v})$ for all odd primes $v\neq p, q$.

Since the Hasse symbol of $f$ at $\Bbb R$ is equal to $(\frac{-1,-1}{\Bbb R})$, one concludes that the Hasse symbol of $f$ at $v=2$ is equal to $(\frac{-1,-1}{\Bbb Q_2})$  by the Hilbert reciprocity law \cite[71:18]{OM}. Therefore $f$ is isotropic over $\Bbb Q_2$ by \cite[58:6]{OM}. By Prop.\ref{univ-ter}, one obtains that $f$ is locally universal over $\Bbb Z_2$. 

By the same proof of Prop.\ref{one-class}, one concludes that $\text{gen} (f)$ contains a single proper class. Therefore each integer which is represented by $f$ locally will be represented by $f$. Since all integers between $-N$ and $N$ are not divisible by $p$ and $q$, these integers are represented by $f$ at primes $v=p$ and $q$. This implies that all integers between $-N$ and $N$ are represented by $f$ as desired. 

Since there are infinitely many such pairs of primes $(p, q)$ as above, one gets infinitely many quadratic forms as above. The proof is complete. 
\end{proof}

\section{Potential universal quadratic forms}

In this section, we consider the following generalization of universal lattices. 

\begin{defi}  Let $k$ be a number field and $L$ be an $\frak o_k$-lattice. 

(1) We say that $L$ is potentially universal if there is a finite extension $K/k$ such that $\frak o_k \subset Q(\widetilde{L})$ with $\widetilde{L}=L\otimes_{\frak o_k} \frak o_K$.

(2) An element $\alpha\in \frak o_k$ is potentially represented by $L$ if there is a finite extension $K/k$ such that $\alpha\in Q(\widetilde{L})$ with $\widetilde{L}=L\otimes_{\frak o_k} \frak o_K$.
\end{defi}

Let $\bar k$ be an algebraic closure of $k$,  $\frak o_{\bar k}$ be the ring of all algebraic integers in $\bar k$ and $\frak o_{\bar k}^\times$ be the group of units in $\frak o_{\bar k}$.
 For any $\delta\in \frak o_{\bar k}$, we set 
$$ \frak o_{\bar k}(\delta)= \{ x\in \frak o_{\bar k}: \ (x, \delta) =\frak o_{\bar k} \} .$$  

\begin{lem} \label{surjective} With the above notations, one has
$$ \frak o_{\bar k}(\delta)= \frak o_{\bar k}^\times + \delta \frak o_{\bar k} $$ for any $\delta\in \frak o_{\bar k}$.
\end{lem} 

\begin{proof} Since  $$ \frak o_{\bar k}(\delta)= \begin{cases} \frak o_{\bar k}^\times  \ \ \ & \text{if $\delta=0$} \\  
\frak o_{\bar k} \ \ \ & \text{if $\delta\in \frak o_{\bar k}^\times$} \end{cases} $$ by the definition of $\frak o_{\bar k}(\delta)$, the result holds obviously for these two cases.

For $\delta\in \frak o_{\bar k}\setminus \frak o_{\bar k }^\times$ with $\delta \neq 0$, we only need to show  $ \frak o_{\bar k}(\delta) \subseteq \frak o_{\bar k}^\times + \delta \frak o_{\bar k} $. For any $\gamma \in \frak o_{\bar k}(\delta)$, one needs to construct two polynomials 
$$f(x)=x^{m}+a_{m-1}x^{m-1}+...+a_{1}x+1 \ \ \ \text{and} \ \ \ g(x)=x^{m}+b_{m-1}x^{m-1}+...+b_{1}x+b_{0}$$
over $\frak o_{\bar k}$ such that  
$${\delta}^{m} g(x) ={\gamma}^m f(\frac{1+\delta x}{\gamma})  \ \ \ \ \ \text{or} \ \ \ \ \  \gamma^m f(x)= \delta^m g(\frac{\gamma x -1}{\delta}).  $$ 
Equivalently, the following linear equations 
 $$
        \left[
        \begin{array}{ccccc}
        (-1)^{0}C^{0}_{m}& 0& 0& \cdots&0\\ 
        
        (-1)^{1}C^{1}_{m}& (-1)^{0}C^{0}_{m-1}& 0& \cdots&0\\
        
        (-1)^{2}C^{2}_{m}& (-1)^{1}C^{1}_{m-1}& (-1)^{0}C^{0}_{m-2}& \cdots&0\\
        
        \vdots& \vdots& \vdots& \ddots& \vdots\\
        
        (-1)^{m}C^{m}_{m}& (-1)^{m-1}C^{m-1}_{m-1}& (-1)^{m-2}C^{m-2}_{m-2}& \cdots&(-1)^{0}C^{0}_{0}
        \end{array}
        \right]
        \left[
        \begin{array}{ccccc}
        1\\
        \delta  b_{m-1}\\
        \delta^{2}b_{m-2}\\
        \vdots\\
        \delta^{m}b_{0}
        \end{array}
        \right]
        =
        \left[
        \begin{array}{ccccc}
        1\\
        \gamma a_{m-1}\\
        \gamma^{2}a_{m-2}\\
        \vdots\\
        \gamma^{m}
        \end{array}
        \right ] $$
 are solvable over $\frak o_{\bar k}$ for some positive integer $m$, where $a_1, \cdots, a_{m-1}$ and $b_0, b_1, \cdots, b_{m-1}$ are variables. 

  Let $K/k$ be a finite extension such that $\delta, \gamma\in K$. Choose a positive even integer $m$ such that $\gamma^m \equiv 1 \mod \delta \frak o_K$. By Chinese Remainder Theorem \cite[21:2]{OM}, the above linear equations are solvable over $\frak o_K$ if and only if they are solvable over all completions of $\frak o_K$. 

If $v\in \Omega_K\setminus \infty_K$ with $\ord_v(\delta)=0$, one obtains that the above linear equations are solvable over $\frak o_{K_v}$ by arbitrarily choosing values of $a_1, \cdots, a_{m-1}$ in $\frak o_{K_v}$ and inverting the matrix. 

If  $v\in \Omega_K\setminus \infty_K$ with $\ord_v(\delta)>0$, one obtains that $\ord_v(\gamma)=0$ by $\gamma\in \frak o_{\bar k}(\delta)$. By taking $$b_0=\cdots=b_{m-2}=0 \ \ \ \text{ and } \ \ \ b_{m-1}= (1-\gamma^m)\delta^{-1} \in \frak o_K , $$ one concludes that the above equations is solvable over $\frak o_{K_v}$. 
\end{proof}

One can reformulate Lemma \ref{surjective} as the following interesting result. 

\begin{cor} For any ideal $\frak a$ of $\frak o_{\bar k}$, the natural map $\frak o_{\bar k}^\times \rightarrow (\frak o_{\bar k}/\frak a)^\times$ is surjective. 
\end{cor}

\begin{proof}  Let $x\in \frak o_{\bar k}$ with  $\overline{x}\in  (\frak o_{\bar k}/\frak a)^\times$. Then $(x)+\frak a = \frak o_{\bar k}$. There is $\delta\in \frak a$ such that $x\in  \frak o_{\bar k}(\delta)$. By Lemma \ref{surjective}, there is $y\in \frak o_{\bar k}^\times$ such that $x-y\in (\delta) \subseteq \frak a$ as desired. 
\end{proof}

The following theorem is known as the principal idea theorem in class field theory (see \cite[Chapter IV, (8.6) Theorem] {Neu}).  

\begin{thm}\label{PIT} Every ideal of $k$ becomes an principal ideal in the Hilbert class field of $k$. 
\end{thm}

By applying Theorem \ref{PIT}, one can obtain the following result. 
\begin{lem}\label{rep} Let $k$ be a number field. If $\delta$ is a non-zero element in $\frak o_k$, then there is a finite extension $K/k$ with $b_1, \cdots, b_n, \sqrt{\delta} \in \frak o_K$ such that any non-zero element $x\in \frak o_k$ can be written as 
$$ x=b_1^{r_1} b_2^{r_2} \cdots b_n^{r_n} (1+2 \rho(x) \sqrt{\delta}) $$ where $r_1\geq 0, \cdots, r_n\geq 0$ are integers and $\rho(x)\in \frak o_K$. 
\end{lem} 

\begin{proof} Let $K_1=k(\sqrt{\delta})$ and $K_2$ be the Hilbert class field of $K_1$. Let $\frak p_1, \cdots, \frak p_m$ be all prime ideals of $\frak o_{K_1}$ containing $2\sqrt{\delta} \frak o_{K_1}$. By , there is $a_i\in \frak o_{K_2}$ such that $\frak p_i\frak o_{K_2}=(a_i)$ for $1\leq i\leq m$. 

Consider the action of additive group $(2\sqrt{\delta}\frak o_{K_2}, +)$ on the set $\frak o_{K_2}\cap \frak o_{\bar k}(2\sqrt{\delta})$ by translation. Since $\delta\neq 0$ and $\frak o_{K_2}/(2\sqrt{\delta} \frak o_{K_2})$ is finite, the set $\frak o_{K_2}\cap \frak o_{\bar k}(2\sqrt{\delta})$ can be decomposed into finitely many orbits
$$ \frak o_{K_2}\cap \frak o_{\bar k}(2\sqrt{\delta}) = \bigcup_{i=1}^t (\xi_i + 2\sqrt{\delta} \frak o_{K_2}) \ \ \ \text{with} \ \ \ \xi_i= \eta_i+ 2\sqrt{\delta} \tau_i \in \frak o_{K_2}\cap \frak o_{\bar k}(2\sqrt{\delta})  $$ by Lemma \ref{surjective}, where $\eta_i\in \frak o_{\bar k}^\times$ and $\tau_i\in \frak o_{\bar k}$ for $i=1, \cdots, t$. 

By Dirichlet Unit Theorem \cite[33:10]{OM}, the group of units $\frak o_{K_2}^\times$ is generated by $u_1, \cdots, u_s$. Let $$K=K_2(\eta_1, \cdots, \eta_t, \tau_1, \cdots, \tau_t) $$ with $$ \{b_1, \cdots, b_n\} = \{ a_1, \cdots, a_m, \eta_1, \cdots, \eta_t, u_1, \cdots, u_s, u_1^{-1}, \cdots, u_s^{-1} \} . $$

For any $x\in \frak o_k$ with $x\neq 0$, one has the decomposition $$x\frak o_{K_1} = \frak a \cdot \prod_{i=1}^m \frak p_i^{l_i} \ \ \ \text{with} \ \ \ (\frak a, 2\sqrt{\delta} \frak o_{K_1})=1 $$ over $\frak o_{K_1}$, where $l_1\geq 0, \cdots, l_m\geq 0$ are integers. Since $\frak a\frak o_{K_2} =a\frak o_{K_2}$ by Theorem \ref{PIT}, one obtains that $$a\in \frak o_{K_2}\cap \frak o_{\bar k}(2\sqrt{\delta})= \bigcup_{i=1}^t (\xi_i + 2\sqrt{\delta} \frak o_{K_2})  . $$ There is $1\leq i_0\leq t$ such that 
$$ a-\xi_{i_0}=a-(\eta_{i_0} + 2\sqrt{\delta} v_{i_0}) \in 2\sqrt{\delta} \frak o_{K_2} . $$ This implies that $a=\eta_{i_0}(1+2\sqrt{\delta} \rho(x))$ with $\rho(x)\in \frak o_K$. Since $$x\frak o_{K_2}=(a \frak o_{K_2}) \prod_{i=1}^m (a_i^{l_i} \frak o_{K_2}) ,$$ there is $u\in \frak o_{K_2}^\times$ such that 
$$ x=u a  (\prod_{i=1}^m a_i^{l_i}) =(\prod_{i=1}^s u_i^{c_i}) (\prod_{i=1}^m a_i^{l_i})  \eta_{i_0} (1+2\sqrt{\delta} \rho(x))  $$ with $c_i\in \Bbb Z$ for $1\leq i\leq s$ as desired.  \end{proof}

The main result of this section is the following theorem. 

\begin{thm} \label{potential}  An $\frak o_k$-lattice $L$ with $\rank (L)\geq 2$ is potentially universal if and only if $1$ is potentially represented by $L$. 
\end{thm}

\begin{proof} Sufficiency. By finite extension of the ground field if necessary, we simply assume that $L$ represents $1$ over $\frak o_k$. There is $x\in L$ such that $Q(x)=1$. Since $\rank(L)\geq 2$, there is $y\in L\cap (kx)^{\perp}$ satisfying $Q(y)\in \frak o_k$, where $(kx)^{\perp}$ is the orthogonal complement of $kx$ in $kL$. Let $$M=(\frak o_k x\perp \frak o_k y) \subseteq L. $$ One only needs to show that $M$ is potentially universal. Write $Q(y)=\delta\in \frak o_k$. By Lemma \ref{rep}, there is a finite extension $F/k$ with $b_1, \cdots, b_n, \sqrt{\delta} \in \frak o_F$ such that any non-zero element $\alpha \in \frak o_k$ can be written as 
$$ \alpha=b_1^{r_1} b_2^{r_2} \cdots b_n^{r_n} (1+2 \rho(\alpha) \sqrt{\delta}) $$ where $r_1\geq 0, \cdots, r_n\geq 0$ are integers and $\rho(\alpha)\in \frak o_F$. Let $K=F(\sqrt{-1},\sqrt{b_{1}},\sqrt{b_{2}},...,\sqrt{b_{n}})$. Then
$$\alpha =[\prod_{i=1}^{n}(\sqrt{b_{i}})^{r_{i}}(1+\sqrt{\delta}\rho(\alpha))]^{2}+\delta[\prod_{i=1}^{n}(\sqrt{b_{i}})^{r_{i}}(\sqrt{-1}\rho(\alpha))]^{2} \in Q(\widetilde{M})$$
where $\widetilde{M}= M\otimes_{\frak o_k} \frak o_K$.

Necessity is obvious. 
\end{proof}

Inspired by the referee's question, we provide another description of potentially universal quadratic forms.

\begin{cor}\label{another} An $\frak o_k$-lattice $L$ with $\rank (L)\geq 2$ is potentially universal if and only if $\frak n(L)= \frak o_k$.
\end{cor}
\begin{proof}  One only needs to show sufficiency. 
We first claim that there is a binary $\frak o_k$ sub-lattice $M$ of $L$ such that $\frak n(M)= \frak o_k$. 
Suppose $\rank (L)\geq 3$. Let $S$ be a finite set of non-archimedean primes of $k$ containing all dyadic primes such that $\ord_v(\det(L))=0$ for $v\not\in (S\cup \infty_k)$. Since $\frak n(L)=\frak o_k$, one can choose $x_v, y_v \in L_v$ such that $$\frak n(\frak o_{k_v} x_v+ \frak o_{k_v} y_v) = \frak o_{k_v}$$ for each $v\in S$.  By the number field version of \cite[Lemma 1.6]{HKK} (see also \cite[\S 3]{HKK}), there are $x, y\in L$ which approximate $x_v, y_v$ at each $v\in S$ respectively such that $$M=\frak o_k x+ \frak o_k y \ \ \ \text{with} \ \ \ \det(M) \in \frak o_{k_v}^\times \ \ \ \text{for all $v\not \in (S\cup \infty_k)$ except one prime $v_0$} $$ and $\ord_{v_0}(\det(M))=1$. Since $\ord_{v_0}(\det(M))=1$, one concludes that $\frak n(M_{v_0})=\frak o_{k_{v_0}}$. 
Therefore the claim follows and one can assume $\rank (L)=2$. By Theorem \ref{potential}, we only need to prove that $1$ is potentially represented by $L$.

Let $K=k(\sqrt{-det(L)})$. Then the quadratic space $kL$ becomes isotropic over $K$. Since $$\frak n(L\otimes_{\frak o_k} \frak o_K) = \frak n(L) \otimes_{\frak o_k} \frak o_K , $$
one can further assume that the quadratic space $kL$ is isotropic. Write $$L=\frak a x + \frak b y \ \ \ \text{ with } \ \ \ Q(x)=0$$ where $\frak a$ and $\frak b$ are fractional ideals of $k$.
By Theorem \ref{PIT}, one can assume that both $\frak a$ and $\frak b$ are principal by scalar extension to the Hilbert class field of $k$. Therefore one can write $$ L= \frak o_k x_1+ \frak o_k y_1 \ \ \ \text{with} \ \ \ Q(x_1)=0, \ \ \ Q(y_1)=a \ \ \ \text{and} \ \ \ B(x_1, y_1)= \frac{b}{2} $$ where $a, b\in \frak o_k$. Since $\frak n(L)= \frak o_k$, one concludes that $(a, b)= \frak o_k$.  
By Lemma \ref{surjective}, there are $u\in \frak o_{\bar k}^\times$ and $\xi \in \frak o_{\bar k}$ such that $a=u+b\xi$. Let $E=k(\xi, \sqrt{u})$. Then $$\frac{\xi}{\sqrt{u}} x_1 - \frac{1}{\sqrt{u}} y_1 \in (L\otimes_k \frak o_E) \ \ \ \text{and} \ \ \ Q(\frac{\xi}{\sqrt{u}} x_1 - \frac{1}{\sqrt{u}} y_1) =1 $$ as desired. 
\end{proof}

\begin{cor} \label{st}  If $L$ is a potentially universal $\frak o_k$-lattice for a number field $k$, then $L\otimes_{\frak o_k} \frak o_K$ is a potentially universal $\frak o_K$-lattice for any finite extension $K/k$. 
\end{cor}
\begin{proof} Since there are only finitely many primes of $k$ ramified in a fixed finite extension of $k$, one obtains that $\rank (L)\geq 2$.  The result follows from that fact $\frak n(L\otimes_{\frak o_k} \frak o_K) = \frak n(L) \otimes_{\frak o_k} \frak o_K $ and Corollary \ref{another}.   
\end{proof}

\bigskip

\noindent\textbf{Acknowledgements.}  We would like to thank the referee for pointing out incompleteness of the proof of Prop.\ref{univ-bi} in the early version of this paper. 
The first named author is supported by NSFC grant no.11631009.

\bibliographystyle{alpha}
%\bibliography{mybib1}
\end{document}